\documentclass[11pt]{article}
\usepackage{graphicx}
\usepackage{authblk}
\usepackage{mathtools}
\usepackage{amsmath}
\usepackage{amsthm}
\usepackage{tikz}
\usepackage{amsmath}
\usepackage{amsfonts}
\usepackage[pagewise]{lineno}

\usepackage{amssymb, amsmath}
\usepackage{epstopdf, wrapfig, framed, caption}
\usepackage{color}
\usepackage[makeroom]{cancel}
\usepackage{stmaryrd}
\usepackage{hyperref}
\usepackage{makecell}
\usepackage[extra]{tipa}

\def\mbbR{\mathbb R}

\newcommand{\eps} {\epsilon}

\newcommand{\be}{\begin{equation}}
\newcommand{\ee}{\end{equation}}
\newcommand{\beq}{\begin{equation}}
\newcommand{\eeq}{\end{equation}}
\newcommand{\ba}[1] {\begin{array}{ #1 }}
\newcommand{\ea}{\end{array}}

\def\tU{\tilde{U}}

\def\mrL{\mathrm L}

\def\cC{{\mathcal C}}
\def\cD{{\mathcal D}}

\def\cL{{\mathcal L}}
\def\cM{{\mathcal M}}
\def\cN{{\mathcal N}}

\def\cS{{\mathcal S}}

\def\mbbC{\mathbb{C}}

\def\mbbE{\mathbb{E}}

\def\mbbR{\mathbb{R}}

\def\mbbZ{\mathbb{Z}}

\def\mrA{\textrm A}
\def\mrB{\textrm B}
\def\mrC{\textrm C}
\def\mrD{\textrm D}

\def\mrF{\textrm F}
\def\mrG{\textrm G}
\def\mrH{\textrm H}
\def\mrI{\textrm I}

\def\mrL{\textrm L}
\def\mrM{\textrm M}

\def\mrO{\textrm O}

\def\mrR{\textrm R}

\def\mrU{\textrm U}
\def\mrV{\textrm V}

\def\mrd{\textrm d}

\newtheorem{thm}{Theorem}
\newtheorem{lemma}[thm]{Lemma}

\newtheorem{remark}{Remark}
\newtheorem{MR}{Main Result}

\renewcommand{\appendix}{\Alph{section}}
\numberwithin{equation}{section}
\title{Curve Lengthening Bifurcations in Modally Filtered Nonlinear Schr\"odinger Systems}

\author[1]{Keith Promislow}
\author[2]{Abba Ramadan}
\affil[1]{Department of Mathematics, Michigan State University,
East Lansing, MI 48824, USA}
\affil[2]{Department of Mathematics,
The University of Alabama, 
Tuscaloosa, AL 35401, USA}

\begin{document}
\maketitle
\begin{abstract}
Extensions of the parametric nonlinear Schr\"odinger equations (PNLS) for phase sensitive optical resonance are developed that preserve the curve lengthening bifurcation seen in the original system. This bifurcation occurs in sharp interface reductions when the motion of the interface transitions from curvature driven flow (curve shortening) to motion against curvature regularized by higher order Willmore effects (curve lengthening). We construct a specific class of down-phase self-interaction operators via a spectral transform of the down-up operator.  While the bifurcation regime the corresponding modally filtered nonlinear Schr\"odinger systems preserve the linear stability of the front, admit the sign flip in the linear term in the normal velocity while preserving the proper sign of the Willmore terms. 
\end{abstract}
 


\section{Overview}
There are many examples of dissipative systems that admit quasi-adiabatic limits described by curvature driven motion of interfaces. The Allen Cahn equation is a seminal example \cite{bib:Ilmanen93, bib:LeeKim15}. The simplest curvature mediated flow is motion by mean curvature (curve shortening), and it admits a curve lengthening bifurcation in which the sign of the curvature coupling changes. This sign shift transitions the system from a curve shortening regime to motion {\bf against} curvature, also called curve lengthening, regularized by higher-order Willmore-type effects, such as curvature surface diffusion. These higher order operators are required for the reduced model to remain locally well posed. There are several examples of curve lengthening flows arising as gradients of system energies \cite{Bellettini10, bib:CP23, bib:DP15}.  Motion by curvature generically reduces front length, leading either to flat fronts or homogeneous single-phase solutions. The curve lengthening bifurcation leads to complex transients that generically lead to self-intersection of the interface, representing an abrupt change in the structure of the long-term attractor of these nonlinear hyperbolic systems, \cite{bib:GT87}.

Prior work  by the authors analyzed the curve lengthening bifurcation in the parametric nonlinear Schr\"odinger (PNLS) system, \cite{bib:PR24}. This is a dispersive system that is reduced from a class of optical parametric oscillator (OPO) systems in a large detuning limit \cite{bib:OPO5, bib:OPO2,  bib:PK,  bib:TOWG, bib:Trillo97}. 
These optical parameter amplifiers use partially mirrored saturating absorbers or phase sensitive amplifiers at each end of the optical cavity to select a preferred phase. The physical cavities often support a several distinguished cavity waves, mode locking synchronizes the multiple modes to obtain resonant effects, \cite{bib:Ham25}.
These systems have gained popularity as tools for the generation of ultra-short pulses \cite{bib:Huss19}. 

The nonlinear Schr\"odinger type evolution models describe the return-map corresponding to the electric field at the partially mirrored end after  passage through the cavity. These systems are naturally posed in 1+2D dimensions. The phase sensitive amplification and damping can be modeled within an NLS type system for a complex valued phase function $u:\Omega\times\mbbR_+\mapsto\mbbC,$
where $\Omega\subset\mbbR^2$ is a bounded domain. The evolution takes a general form
\beq \label{e:GenModeFilter}
u_t+ \mrA(|u|^2)u+\mrB(|u|^2)u^*=0, \eeq
where, given $u$, $\mrA(|u|^2)$ and $\mrB(|u|^2)$ are intensity dependent linear operators that model the interaction of the field $u$ and its complex conjugate $u^*.$  The phase invariance that generically accompanies propagation in NLS systems has been broken by the complex conjugate term, promoting the stability of standing waves.

\begin{figure}[t]
\vspace{-0.9in}
\centering
\includegraphics[height=1.8in]{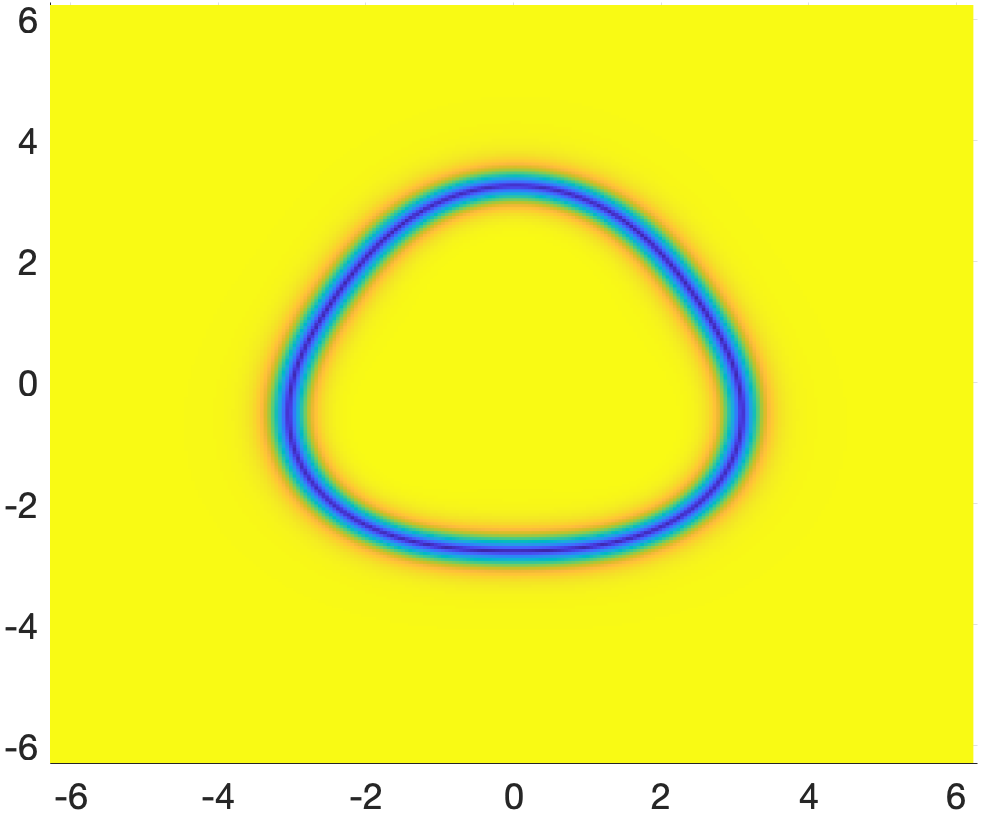}
\includegraphics[height=1.8in]{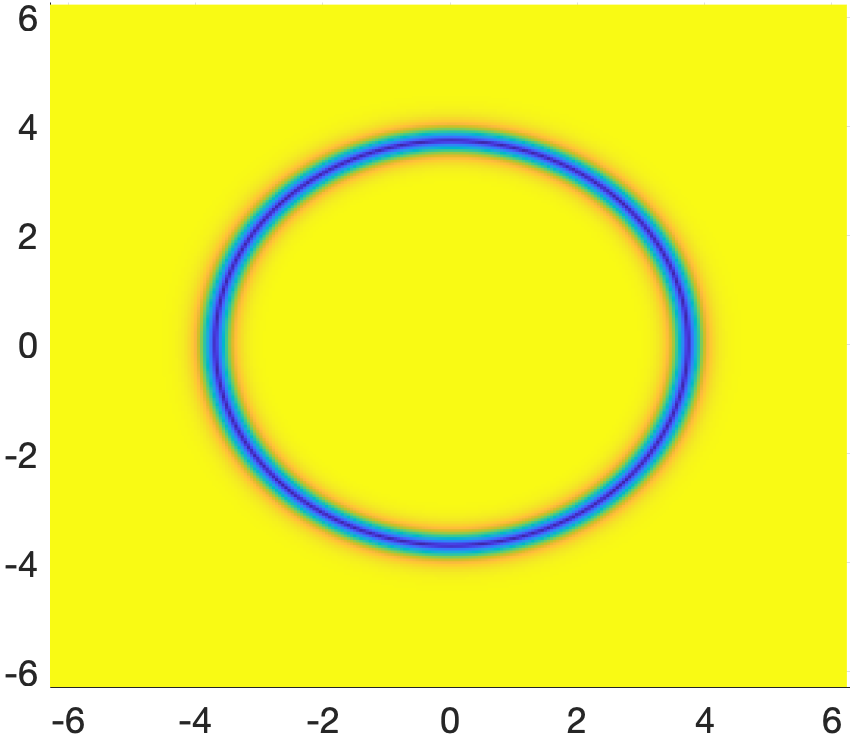} 
\includegraphics[height=1.8in]{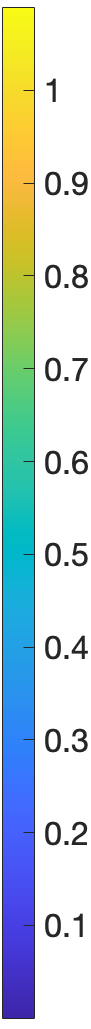}
 \put(-315, 12){\large\textbf{$t=0$}}
  \put(-155, 12){\large\textbf{$t=10^4$}} \\[5mm]
\includegraphics[height=1.6in]{Mu-503T0c.png}
\includegraphics[height=1.6in]{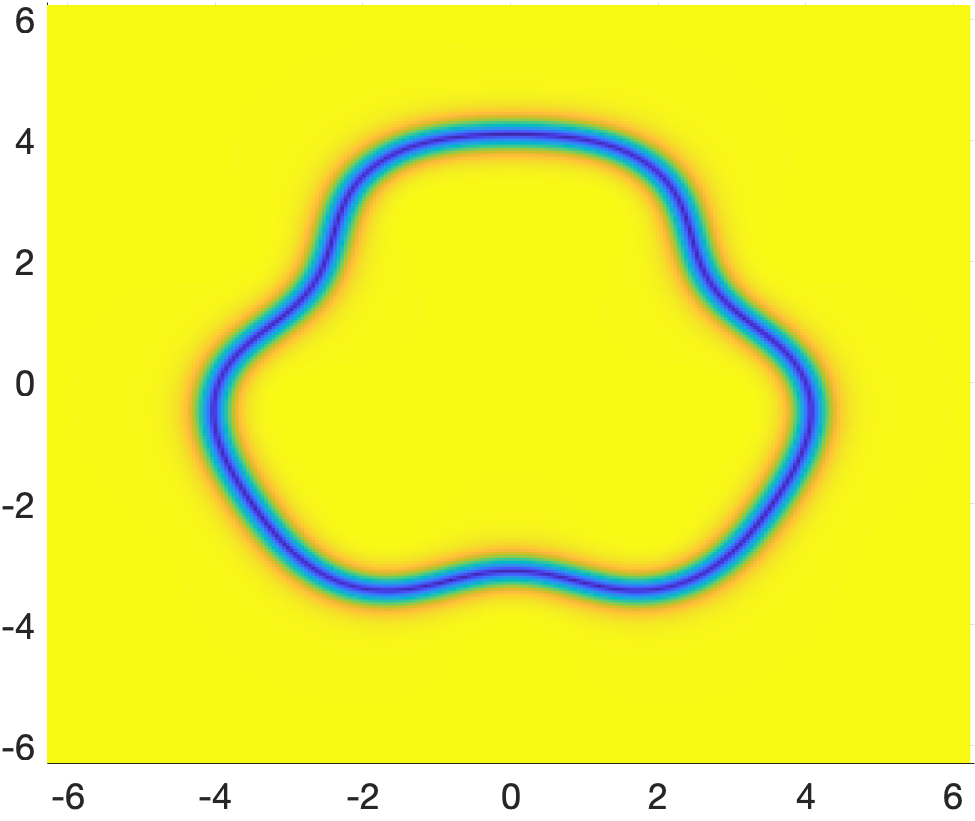}
\includegraphics[height=1.6in]{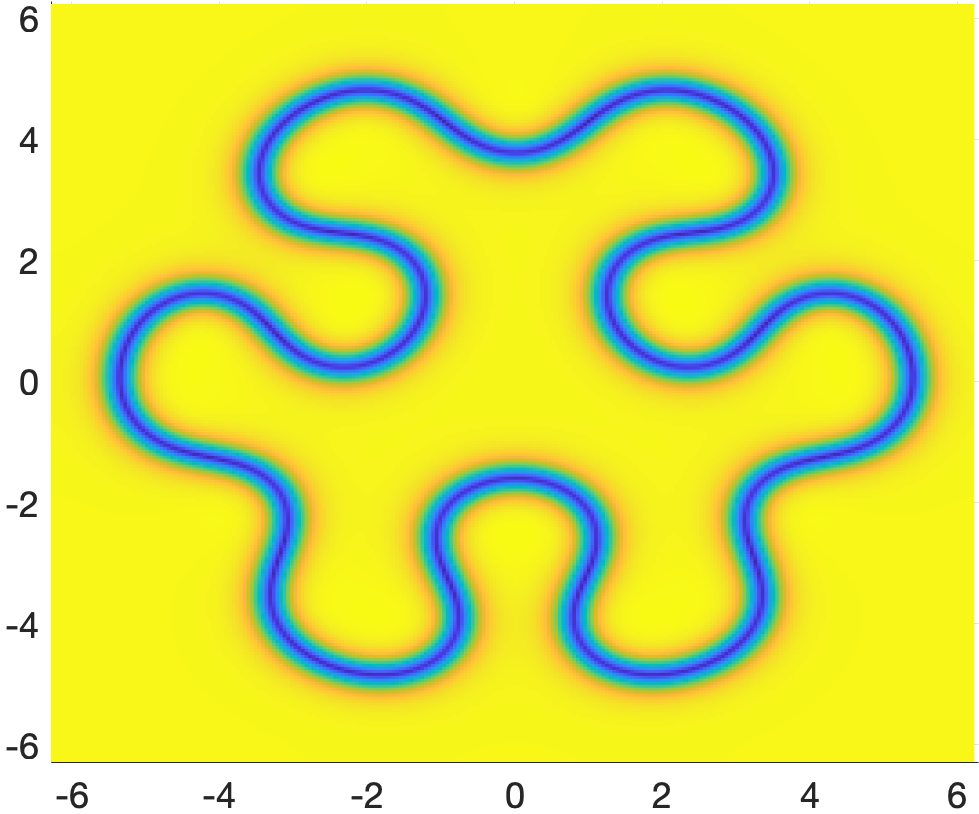}
 \put(-415, 10){\large\textbf{$t=0$}}
  \put(-271, 10)
  {\large\textbf{$t=310$}}
   \put(-131, 10){\large\textbf{$t=520$}}
 \\
\caption{Contour plots of the modulus $|U|$ over $[-2\pi,2\pi]^2$ simulated from the PNLS version of \eqref{e:MFPNLS} from the same initial data. The top row  corresponds to parameters that induce motion by curvature $\mu>0$. In the bottom row the parameters induc motion against curvature ($\mu<0)$. Reprinted with permission from \cite{bib:PR24}.}
\label{f:mu>0}
\end{figure}
We investigate the features of the maps $\mrA$ and $\mrB$ in \eqref{e:GenModeFilter} that lead to a curve lengthening bifurcation associated to a front solution.  When normalized so that the selected phase is purely real, the system can instructively be written as a vector system for the real and imaginary components
\beq\label{e:MFPNLS}
U_t =\mrF(U):=\begin{pmatrix} 0 & \cN_-(|U|^2;\mu)  \cr
          -\cN_+(|U|^2)& -\cM(|U|^2) \end{pmatrix} U,\eeq
with a bifurcation parameter $\mu\in\mbbR$ included in $\cN_-.$
For $\cM=0$ this is a classical NLS type system and for $\cM=\beta \mrI$ with $\beta>0$ this is a PNLS type system. The operator families $\cN_\pm$ are nonlinear maps from $U\in L^2(\Omega)$ into linear operators from $H^2(\Omega)$ to $L^2(\Omega)$. We consider self-adjoint, second-order operators that balance dispersion and nonlinearity, \beq\label{e:cN_expression}
\begin{aligned}
\cN_-(U;\mu)&= -\eps^2 \Delta +g_-(|U|^2;\mu),\\
\cN_+(U)&=-\eps^2\Delta + g_+(|U|^2),
\end{aligned}
\eeq
via smooth functions $g_\pm:\mbbR\mapsto\mbbR$  that model the nonlinear features of the cross-talk between real and imaginary components. With this normalization the operators are predominantly positive, with finite dimensions of negative eigenspaces.  A primary assumption is that the system supports a 1D heteroclinic (front) solution $\phi$ with odd parity in the `up' equation, solving
\beq\label{e:phi-eq} \cN_+(|\phi|^2)\phi:=\left(-\partial_z^2+g_+(\phi^2)\right)\phi=0,\eeq
in a rescaled $z=x/\eps$ 1D variable. This solution forms the ``up'' front profile
\beq\label{e:up-front}
\Phi(z)=\begin{pmatrix}
    \phi(z)\cr 0
\end{pmatrix}.
\eeq
Without loss of generality we assume $\phi$ connects $1$ to $-1.$  The 1D construction is relevant as the Frenet variables  associated to an interface in $\Omega$ convert position $x\in\Omega$ to scaled distance $z=z(x)$ to the interface. The Frenet variables map a curved interface into a flat interface.  We assume that both $\cN_\pm$ have positive far field limits, with $g_{\pm}(1)>0,$ uniformly for $|\mu|$ sufficiently small. We assume that $\phi$ is the ground state of its defining system. More specifically, we assume the linearization of \eqref{e:phi-eq} about $\phi$ yields a non-negative operator with kernel spanned by the symmetry eigenfunction $\phi'$.  We assume that the self-adjoint operator $\cN_-(\phi;\mu)$ is strictly positive for $\mu>0$ and is positive save for a  one-dimensional negative ground state for $\mu<0.$ The arrival of the negative mode in $\cN_-$ is the mechanistic trigger of the curve lengthening bifurcation.

\subsection{Modally Filtered Schr\"odinger Equation (MPSE)}

We identify a class of linear operators $\cM$, generalizing the identity, that preserve the curve lengthening bifurcation.
The operator $\cM$ breaks the phase invariance, introducing a self-interaction for the `down'  or purely imaginary, component of the wave. 
The curve lengthening bifurcation requires that the curvature motion bifurcates with the change
in the negative index of $\cN_-(\phi^2;\mu)$ while the  front $\Phi$ remains linearly stable.  These are competing effects that require cooperation between the operators $\cN_-(\phi^2)$ and $\cM(\phi^2).$ We identify a family of operators $\cM(U)$ that preserve the structure of this bifurcation. The operators are self-adjoint, uniformly positive, bounded, and positivity preserving with respect to $\cN_-^{-1}(U)$. More specifically, given a finite codimension subspace $S\subset\cD(\cN_-)$,  if the constrained bilinear form induced by $\cN_-$ at the front profile, is positive 
\[\left\langle \cN_-^{-1}(\phi^2)u, u\right\rangle>a\Vert u\Vert_{L^2}^2,\]
for all $u\in S$ and some $a>0$ (see \eqref{e:ConstraintOp}), then positivity preservation implies that the constrained bilinear from associated to the product $\cN_-^{-1}\cM$ also satisfies
\[\left\langle \cN_-^{-1}(\phi^2)\cM(\phi^2) u,u\right\rangle >a\Vert u\Vert^2_{L^2},\]
for all $u\in S.$

 Positively preservation is easiest to deduce when the operators $\cN_-(\phi^2)$ and $\cM(\phi^2)$ commute, and hence the operators have the same eigenmodes.  While this is a strong assumption, the spectral mapping theorem leaves a relatively broad choice of operators $\cM(|U|^2;\mu)=\cS(\cN_-(|U|^2;\mu))$ in terms of smooth functions $\cS:\mbbR\mapsto \mbbR$ that have extensions to  meromorphic functions $\cS:\mbbC\mapsto\mbbC$. The spectral mapping theorem extends $\cS$ to a map on self-adjoint linear operators. 
 We call a system in the form \eqref{e:MFPNLS} a Modal PNLS system if $\cN_\pm$ and  $\cM$ satisfy these assumptions. The word modal is motivated by the fact that the down-up (imaginary to real) and down-down (imaginary to imaginary) filters share eigenmodes. 
 
 The primary result is to establish conditions on $\cS$ such that associated modal PNLS systems preserve the curve lengthening bifurcation. To this end we fix $\mu_*>0$ and assume there exists $a_\pm\in\mbbR$ for which 
 \beq\label{e:g_lower}
 \inf_{s\in\mbbR_+}g_\pm(s)> a_\pm,
 \eeq
 uniformly for $|\mu|\leq \mu_*$. This implies that $\sigma(\cN_-(|U|^2;\mu))\subset [a_-,\infty)$ for all $|\mu|<\mu_*$ and all $U\in L^2(\Omega).$  Our main result is the following.
 \begin{MR}
 \label{MR:1}
Assume the conditions on $\cN_\pm$ above hold. 
Suppose that exists $0<\beta_-<\beta_+<\infty$ such that the spectral map $\cS$ satisfies
\begin{align}
\label{e:cC-bds}
 &\cS:[a_-,\infty) \mapsto [\beta_-,\beta_+],\\
    \label{e:cC-nondecr}
    &\cS'(s) \geq 0, \hspace{0.2in} \textrm{for}\,\textrm{all}\, s\in[a_-,\infty),
 \end{align}
 and $\cS$ has an analytic extension to a complex neighborhood of $[a_-,\infty)$. Then there exists $\mu_*>0$ such that the system \eqref{e:MFPNLS} satisfies the curve lengthening bifurcation in $\mu$ for $|\mu|<\mu_*$. As shown in Theorem\,\ref{t:1D_coercive}, the front solution $\phi$ is linearly stable within the 1D system and $\Phi(\cdot;\gamma)$-front solutions in the 1+2D system evolve through the normal velocity
\beq\label{e:nV_expression}
\mrV=-\alpha_1 \kappa_0 +\eps^2(\nu \Delta_s \kappa_0 +\alpha_3 \kappa_0^3) + O(\eps^3),
\eeq
in the limit $\eps\to 0^+.$ Here $\kappa_0$ is the curvature and $\Delta_s$ is the Laplace-Beltrami operator associated to the interface, and $\nu=\nu(\mu)$, given by \eqref{e:nu} is uniformly positive on $|\mu|<\mu_*$. The coefficient $\alpha_1=\alpha_1(\mu)$, defined in \eqref{e:alpha0}, is smooth in $\mu$ and changes sign from positive to negative as $\mu$ decreases through zero.
\end{MR}
The essential analytic point is that the map $\cM$ is positivity enhancing with respect to $\cN_-^{-1}$ due to the monotonicity of the map $\cS.$ The linear stability is based upon analysis of rational produces to bilinear forms acting on constrained spaces. This is combined with a formal matched asymptotic expansion of the 1+2D dispersive system about the front profile in Frenet variables. A matching to a trivial outer solution yields the normal velocity expression \eqref{e:nV_expression}.  
 The global constraints on $\sigma(\cN_-)$ and the assumptions on $\cS$ allow a spectral representation for $\cM(|U|^2),$
\beq\label{e:Lap-M}
\cM(|U|^2):=\cS(\cN_-(|U|^2;\mu))= \frac{1}{2\pi i} \int_C \cS(\lambda) (\lambda-\cN_-(|U|^2))^{-1} \, \mrd \lambda,\eeq
where the contour $C\subset\mbbC$ lies within the complex domain of analyticity of $\cC$ and contains $\sigma(\cN_-(|U|^2))$ but none of the singularities of the complex extension of $\cS.$ 

\begin{remark}
Beyond the example $\cS\equiv 1$ in the PNLS system, there are many non-trivial examples of spectral maps $\cS$ that satisfy the conditions outlined in the Main Result. These include the  rational polynomial
\[ \cS(s)=\beta_-+(\beta_+-\beta_-)\frac{s}{2a_-+s},\hspace{0.5in} s\in(-2a_-,\infty), \]
which is increasing with the required range, and has a natural  meromorphic extension to all of $\lambda\in\mbbC$ with a simple pole at $\lambda=-2a_-.$ Similarly hyperbolic functions of the form 
\[ \cS(s)=\beta_-+(\beta_+-\beta_-)\frac{e^s}{1+e^s},\hspace{0.5in} s\in\mbbR, \]
have the prescribe range and are increasing on $\mbbR$. Its extension to $\mbbC$ is analytic except for simple poles at $\lambda=(2k+1)\pi i$ for $k\in\mbbZ.$
\end{remark}
\subsection{Notation}
A prime $\prime$ acting on a function of one variable denotes the derivative with respect to that variable, e.g.
\[ \phi'(z) =\frac{\mrd\phi}{\mrd z}.\]
The usual $L^2$ complex inner product over the real line, $\mbbR$ is denoted
\beq \label{e:L2-inner}
\begin{aligned}
\langle u, v\rangle&=\int_\mbbR u(z)v^*(z)\,\mrd z,
\end{aligned}
\eeq
where $*$ is complex conjugation. The associated norm
$\Vert u\Vert_{L^2}=\sqrt{\langle u, u\rangle}$.

\section{Spectral Analysis in 1D}

From the form of the front solution \eqref{e:up-front}, the linearization of $\mrF$ at $\Phi$ is the $2\times 2$ matrix of linear operators 
\beq
\label{e:cL-def}
\cL  = \begin{pmatrix}
0 & \cN_-(\Phi;\mu)  \cr
-\cN_+'(\Phi)    & -\cM(\Phi) \end{pmatrix},
\eeq
where we have introduced
\[   \cN_+'(\Phi) :=\cN_+(\Phi)+2g_+'(\phi^2). \]
We restrict to 1D, rescale $z=x/\varepsilon$ and consider the problem on the unbounded domain $\mbbR$. 
When restricted to 1D the linearized operator $\cL$ reduces to
\be
\label{e:mrL-def} 
\mrL = \begin{pmatrix} 0 & \mrD \cr -\mrC& -\cS_\mrD\end{pmatrix},
\ee
with the 1D restrictions of $-\cN_-(\phi^2)$, $\cN_+'(\phi^2),$ and $\cM(\phi^2)$ denoted by
\be \label{e:CD-def}
\begin{aligned}
  \mrC &= -\partial_z^2 +g_+(\phi^2),\\
  \mrD &= -\partial_z^2 +g_-(\phi^2;\mu)\\
\cS_\mrD&:=\cS(\mrD).\end{aligned}
  \ee
   Since $\phi$ is an odd-parity heteroclinic solution of \eqref{e:phi-eq} on $\mbbR$ operator $\mrC$ has an even parity translational kernel spanned by $\phi'$, while  $\phi^2$ has even parity. The assumptions on $\cN_\pm(\phi^2)$ impose the following features on the spectrum of the self-adjoint operators $\mrC$ and $\mrD.$
There exists $\omega,\mu_*>0$ such that for all $|\mu|<\mu_*$ we have
\be
\ba{rcl} \{\Re\lambda<\omega\}\cap\sigma_p(\mrC)&=&\{ (0,\phi^\prime)\},\\
 \{\Re\lambda<\omega\}\cap \sigma_p(\mrD)&=&\{(\lambda_\mrD(\mu),\psi(\mu))\},\ea
 \label{e:CDspec}.\ee
 where $\lambda_\mrD(\mu)$ is smooth in $\mu$, positive for $\mu>0$ and negative for $\mu<0.$
 The ground state eigenfunction $\psi$ for $\mrD$ is non-negative and scaled to have $\Vert \psi\Vert_{L^2}^2=1.$
 
 \begin{lemma}\label{l:Kernel_L}
There exists $\mu_*>0$ such that kernel of $\mrL$ is simple for all $|\mu|\leq \mu_*.$  For $\mu\neq0$ the kernel of $\mrL$ and of its adjoint $\mrL^\dag$ are spanned by the vectors
\be \Psi_0 = \begin{pmatrix}\phi^\prime \cr 0\end{pmatrix}, \quad\quad
    \Psi_0^\dag = \begin{pmatrix} \mrD^{-1}\cS_\mrD\phi^\prime \cr \phi^\prime\end{pmatrix},
\label{e:Psi0}\ee
respectively. For $\mu=0$ the kernels of $\mrL$ and its adjoint are spanned by 
 \be \Psi_0 = \begin{pmatrix}\phi^\prime \cr 0\end{pmatrix}, \quad\quad
    \Psi_0^\dag = \begin{pmatrix}\psi \cr 0 \end{pmatrix}.
\label{psi0mu0}\ee
\end{lemma}
\begin{proof}
    
The simplicity of the kernel is equivalent to the solvability of $\cL^2 P=0$ with $P\neq \Psi_0.$ For $\mu\neq 0$ the Fredholm conditions for solvability yield  non-trivial solutions only if $\langle \cS_\mrD\mrD^{-1}\phi', \phi'\rangle =0.$
For $|\mu|$ small we have the asymptotic inverse formulas
\be
\begin{aligned}
\label{e:Dinv} 
\mrD^{-1} \phi'&= \frac{1}{\lambda_\mrD}\frac{\langle \phi',\psi\rangle}{\Vert\phi'\Vert_{L^2}}\psi+O(\mu^0),\\
\mrD^{-1}\cS_\mrD \phi'&= \frac{\cS(\lambda_\mrD)}{\lambda_\mrD}\frac{\langle \phi',\psi\rangle}{\Vert\phi'\Vert_{L^2}}\psi+O(\mu^0),
\end{aligned}
\ee 
where $O(\mu^0)$ indicates terms that remain uniformly bounded as $\mu\to 0.$ We deduce that
\[ \langle \cS_\mrD\mrD^{-1}\phi', \phi'\rangle =\frac{\cS(\lambda_\mrD)\langle \phi',\psi\rangle^2}{\lambda_\mrD\Vert\phi'\Vert_{L^2}}+O(\mu^0).\]
 Since both $\phi'$ and $\psi$ are non-negative ground state eigenfunctions, their inner product cannot be zero. 
\end{proof}
When scaled to have unit norm, the eigenfunctions of $\mrL$ have smooth dependence on $\mu.$
The inverse of $\mrL$ is given by
\be \label{e:Linv-gen} \mrL^{-1}=\begin{pmatrix}- \mrC^{-1} \cS_\mrD \mrD^{-1} & -\mrC^{-1}\cr
                        \mrD^{-1} & 0\end{pmatrix}.\ee

\subsection{Essential Spectrum of $\mrL$}
Consider the asymptotic operator $\mrL_\infty$ such that $\mrL-\mrL_\infty$ is relatively compact compared to $\mrL_\infty.$
If we assume that $\phi^2\to 1$ as $z\to\pm\infty$, then the associated limiting operators are
\[\begin{aligned} 
\mrC_\infty&=-\partial_z^2+g_+(1),\\
\mrD_\infty&= -\partial_z^2 +g_-(1;\mu),\\
\cS_\infty&=\cS(\mrD_\infty).
\end{aligned}\]
A function $W=e^{\lambda t+ikx}V$ for some $V\in\mbbR^2,$ and $\lambda\in\mbbC$ solves the linear system
\[ \partial_t W = \mrL_\infty W, \]
when $(\lambda,k)$ satisfy the dispersion relation
\[ \det 
\begin{pmatrix} 
-\lambda & k^2 +g_-(1;\mu) \cr
-k^2 -g_+(1) & \cS(k^2+g_-(1,\mu))-\lambda \end{pmatrix}=0.\]
This defines the Fredholm boundaries $\lambda_\mrF$ that comprise the essential spectrum of $\mrL_\infty,$
\[ \lambda_\mrF(k) := -\frac12\left( \cS\Bigl(k^2+g_-(1)\Bigr)\pm \sqrt{\left(\cS\bigl(k^2+g_-(1)\bigr)\right)^2-4(k^2+g_+(1))(k^2+g_-(1))}\right).\]
The Fredholm boundaries coincide with the essential spectrum  $\sigma_{\rm ess}(\mrL_\infty)=\{ \lambda_{\mrF}(k)\,
\bigl |\, k\in\mbbR\} $. This complex set has a genuinely complex component satisfying 
\[ \Re\lambda_\mrF(k) =-\frac12 \cS(k^2+g_-(1)),\] 
and a purely real component that satisfies
\[ \lambda_\mrF(k)\leq -\frac{(k^2+g_+(1))(k^2+g_-(1))}{\cS(k^2+g_-(s))}.\] 
From the bounds \eqref{e:cC-bds} on $\cS$ we deduce that $\sigma_{\rm ess}(\mrL_\infty)\subset \{\Re \lambda\leq -\lambda_{\rm M, ess}\},$ where  
\beq\label{e:Lam_Mess} \lambda_{\rm M, ess} := \min\left\{\frac{\beta_-}{2}, \frac{g_+(1)g_-(1;\mu)}{\beta_+}\right\}>0.\eeq 
In particular, the Fredholm boundaries reside strictly in the negative real part complex plane.
If $\cS=\beta$ is a multiple of the identity, then this reduces the bounds in \cite{bib:PR24} with $\beta_-=\beta_+=\beta>0$.

\begin{lemma}
\label{l:Weyl}
The operator $\mrL$ is a relatively compact perturbation of $\mrL_\infty.$ The two operators have the same essential spectrum, in particular 
\beq \label{e:Weyl}
\sigma_{\rm ess}(\mrL)\subset \left\{\Re \lambda \leq - \lambda_{\rm M, ess}\right\}.
\eeq
\end{lemma}
\begin{proof}
We apply the Weyl essential spectrum theorem. This requires that
\beq\label{e:Weyl1}
(\mrL_\infty-\mrL)(\mrL_\infty-\gamma)^{-1}:L^2(\mbbR)\mapsto H^2(\mbbR),
\eeq
is a compact operator. 
The resolvent of $\mrL_\infty$ has the form
\beq\label{e:Linf_resolv}
(\mrL_\infty-\gamma)^{-1}=
\begin{pmatrix}
-\mrH_\infty^{-1}(\cS_\infty+\gamma)\mrD_\infty^{-1} & 
-\mrH_\infty^{-1}\cr 
\mrD^{-1}_\infty\left(1+\gamma\mrH_\infty^{-1}(\cS_\infty+\gamma)\mrD^{-1}_\infty\right)
&
-\gamma \mrD^{-1}_\infty \mrH_\infty^{-1} 
    \end{pmatrix},\eeq
where we have introduced
\beq\label{e:mrH}
\mrH_\infty:=\mrC_\infty +\gamma(\cS_\infty+\gamma)\mrD^{-1}_\infty.
\eeq
Each of the operators has a Fourier multiplier representation. If $\gamma>0$ is sufficiently large then the Fourier multipliers are each strictly positive. The multipliers also grow quadratically in the Fourier parameter as it tends to infinity. These properties make it straightforward to see that each entry of the resolvent is a bounded map from $L^2(\mbbR)$ into $H^2(\mbbR).$ Since the domain $\mbbR$ is unbounded, $H^2(\mbbR)$ is not compactly embedded in $L^2(\mbbR).$ However compactness is recovered for products of operators of the form
$\mrB\, (\mrL_\infty-\gamma)^{-1},$
where $\mrB$ is a smooth $L^\infty$-bounded multiplier operator that decays to zero at infinity sufficiently fast to lie in $L^1(\mbbR)$, see \cite{bib:KP13}[chapter 3].
The differential
\[\mrL_\infty -\mrL=\begin{pmatrix} 0 & \mrD_\infty-\mrD\cr
\mrC_\infty-\mrC & \cS(\mrD_\infty)-\cS(\mrD)\end{pmatrix},\]
is composed of three operators, two of which can be represented as bounded multipliers that satisfy the decay condition since $\phi(z)\to 1$ at an exponential rate at $z\to\pm\infty.$
It remains to characterize the difference of the functional map.  Returning to the spectral  integral representation
\[\begin{aligned}
    \cS(\mrD_\infty)-\cS(\mrD) &= \frac{1}{2\pi i }\int_C \cS(\lambda)\left( (\lambda-\mrD_\infty)^{-1}-(\lambda-\mrD)^{-1}\right) \,\mrd \lambda,\\
    &= \frac{1}{2\pi i }\int_C \cS(\lambda)\frac{\mrD_\infty-\mrD}{(\lambda-\mrD_\infty)^{-1}(\lambda-\mrD)^{-1}} \,\mrd \lambda,\\
 &=   (\mrD_\infty-\mrD)\mrG,
\end{aligned}\]
where $\mrG:L^2(\mbbR)\mapsto H^2(\mbbR)$ is bounded. Since the composition of compact operators with bounded operators yields compact operators, we deduce that the operator in \eqref{e:Weyl1} is compact and $\sigma_{\rm ess}(\mrL)=\sigma_{\rm ess}(\mrL_\infty)$. 
\end{proof}

\subsection{Point Spectrum of $\mrL$}

 To each $\lambda\in\mbbC$ we associate a constraint space that is co-dimension one in $L^2(\mbbR),$
  \beq\label{e:S-def} 
 S_\lambda:=\{s_\lambda\}^\bot
 \hspace{0.5in}\textrm{for}\hspace{0.5in}
 s_\lambda:=\mrD^{-1}(\cS_\mrD +\lambda)\phi'.\eeq
 The components of eigenfunctions of $\mrL$ associated to non-zero point spectrum lie in constraint spaces. 
 \begin{lemma}
 \label{l:Ptspec_constraint}
 Let $\lambda\in\sigma_p(\mrL)$ with eigenfunction $P=(P_1,P_2)$. If $\lambda\neq 0$ then $P_1$ lies in the associated constraint space: $P_1\in S_\lambda.$  As $|\lambda|\to \infty$ then $S_\lambda\to \{\mrD^{-1}\phi'\}^\bot.$
 \end{lemma}
\begin{proof} The eigenfunctions $P$ associated to point spectrum $\lambda\in \sigma_p(\mrL)$ lie in $L^2(\mbbR)$, and solve
\be \label{e:ptspectrumeq} \lambda P = \begin{pmatrix} 0 &\mrD\cr -\mrC & -\cS_\mrD\end{pmatrix}P.\ee
Assuming that $\lambda\neq 0$ then $P\perp \Psi_0^\dag.$  For $\mu\neq 0$ we combine the orthogonality condition on $P$ with \eqref{e:ptspectrumeq} to deduce that 
$$ P_1\perp \mrD^{-1}(\cS_\mrD+\lambda)\phi' =s_\lambda \hspace{0.5in}\textrm{and}\hspace{0.5in} P_2\perp (\cS_\mrD+\lambda)\phi'.$$
The space $S_\lambda$ is equivalently defined through the scaled version of 
\[s_\lambda=\mrD^{-1}\left(1+\frac{1}{\lambda}\cS_\mrD\right)\phi'.\]
Since $\phi'$ is fixed, the convergence of $S_\lambda$ as $\lambda\to\infty$ is self-evident.
\end{proof}

The operator $\mrL$  is not self-adjoint and its eigenfunctions do not satisfy a direct variational characterization. Nonetheless, the point spectrum can be localized via a nonlinear product of self-adjoint bilinear forms. More specifically,  eliminating $P_2$ from the eigenvalue problem \eqref{e:ptspectrumeq} and taking the complex inner product with $P_1=P_{11}+iP_{12}$ yields the scalar eigenvalue relation 
\begin{equation}\label{e:S_EVP}
\lambda^2\left\langle \mrD^{-1}P_1,P_1\right\rangle+\lambda\left\langle \cS_\mrD \mrD^{-1}P_1,P_1\right\rangle+\left\langle \mrC P_1,P_1\right\rangle=0.
\end{equation}
This is a quadratic equation for $\lambda$ with real coefficients  since $\mrC$, $\mrD$, and $\cS_\mrD \mrD$ are each self-adjoint, real operators.  The quadratic formula characterizes the eigenvalues in terms of bilinear forms,
\beq \label{e:Holyland}
\lambda=-\frac{\left\langle \cS_\mrD \mrD^{-1}P_1,P_1\right\rangle}{2\langle \mrD^{-1}P_1,P_1\rangle}\left(1\pm \sqrt{1-4\frac{\langle \mrD^{-1}P_1,P_1\rangle \langle \mrC P_1,P_1\rangle}{\left\langle \cS_\mrD \mrD^{-1}P_1,P_1\right\rangle^2}}\right).
\eeq
The goal is to characterize settings when the spectrum of $\mrL$, apart from the translational mode at $\lambda=0$, has strictly negative real part. Since $\mrC\geq 0$ the bilinear form
$\langle \mrC P_{11},P_{11}\rangle+\langle \mrC P_{12},P_{12}\rangle\geq 0.$ 
This implies that if $\langle \mrD^{-1}P_1,P_1\rangle<0$ then $\lambda$ is real but can be positive, while if $\langle \mrD^{-1}P_1,P_1\rangle>0$ then $\lambda$ has negative real part. Control of the point spectrum of $\mrL$ resides with control of the bilinear forms. 

One tool is the spectral representation theorem for self-adjoint operators. Let $\mbbE_\rho$  denote the spectral projection of $\mrD$  associated to $\rho\in\sigma(\mrD)$. 
The spectral representation theorem yields the integral representation
$$\mrD^{s}\cS_\mrD P_1 =\int_{\sigma(\mrD)}\cS(\rho) \rho^s\, \mbbE_\rho(P_1)\mrd\rho, $$
where the spectral measure $\mrd\rho$ has delta functions at the point spectrum of $\mrD,$ see \cite{bib:CodLev55}[Theorem 3.2, Chapter 9, Section 3]. In particular, the bilinear form admits the representation
$$\langle\mrD^{s}\cS_\mrD P_1,P_1 \rangle = \int_{\sigma(\mrD)} \rho^s \cS(\rho)|\hat P_1(\rho)|^2\mrd\rho, $$
where $\hat P_1$ represents the mass of the projection of $P_1$ onto the $\rho$ spectral space of $\mrD$, akin to a Fourier coefficient. These satisfy a Plancherel equality
\[ \Vert P_1\Vert_{L^2}^2 = \int_{\sigma(\mrD)} |\hat P_1(\rho)|^2\,\mrd \rho.\]

A second tool is the characterization of the spectrum of a constrained operator. For a self-adjoint operator $\mrO:\cD(\mrO)\mapsto L^2(\mbbR)$ the constrained operator associated to $S_\lambda$ is the operator induced by the bilinear form $\langle \mrO v, v\rangle,$ restricted to act on  $v\in S_\lambda$, see \cite{bib:KP13}[chapter 5]. The constrained operator satisfies
  \beq\label{e:ConstraintOp}
  \mrO\bigl|_{S_\lambda}:= \Pi_\lambda \mrO,
  \eeq 
  where $\Pi_\lambda$ is the orthogonal projection onto $S_\lambda.$ The constrained operator should be understood as a map $ \mrO\bigl|_{S}: S\cap \cD(\mrO) \mapsto S\subset L^2(\mbbR).$ 
  We denote the dimension of the negative eigenspace of a self-adjoint operator $\mrO$ by $n(\mrO)$. Proposition 5.3.1 of \cite{bib:KP13}, see also \cite{bib:KP12}, implies that if $n(\mrO)$ is finite,
  $\mrO$ has no kernel,
  and $S^\bot$ has orthonormal basis $\{s_1,\ldots s_N\}$, then
  \beq \label{e:KP_index}
  n\left(\mrO\bigl|_{S}\right)= n(\mrO)-n(\mrA),
  \eeq
  where $\mrA$ is the $N \times N$ matrix with entries $\mrA_{ij}=\langle \mrO^{-1} s_i, s_j\rangle. $  Some care must be made in applying this result to functions of operators. While the constrained operator $\mrO\bigl|_{S}$ is self-adjoint, the functional mapping of the operator does not commute with the act of constraining. That is $f(\mrO)\bigl|_{S^\bot}\neq f\left(\mrO\bigl|_{S^\bot}\right).$ In particular $\mrO\bigl|_{S^\bot}> 0$ implies $\left(\mrO\bigl|_{S^\bot}\right)^{-1}> 0$ but does not necessarily imply that $\mrO^{-1}\bigl|_{S^\bot}\geq0. $ 
 
  The next Lemma uses the first tool and the assumptions on $\cS$ to establish that the genuinely complex point spectrum of $\mrL$ lie uniformly in the left-half complex plane.

  \begin{lemma}
\label{l:complex-bound}
    Let $\mrD$ satisfy the spectral properties \eqref{e:CDspec} and the map $\cS$ satisfy \eqref{e:cC-bds} with constants $\beta_\pm,$ 
    then the genuinely complex point spectrum of $\mrL$ satisfy
      \beq
      \label{e:Spec1} \sigma_{\rm pt}(\mrL)\backslash\mbbR\subset
      \left\{\Re\lambda<-\frac{\beta_-}{2} \right\}.
      \eeq
    \end{lemma}
\begin{proof}
Let $\lambda\in\sigma_{\rm pt}(\mrL)$ have non-zero imaginary part with eigenvector $(P_1,P_2)$.  From the quadratic formula \eqref{e:Holyland} it follows that $\langle \mrD^{-1}P_1,P_1\rangle>0 $ and 
\[\Re\lambda=-\frac{\langle\mrD^{-1}\cS_\mrD P_1, P_1\rangle}{\langle\mrD^{-1}P_1,P_1\rangle}.\]
We break the spectral integral representation into components associated to positive and negative parts of $\sigma(\mrD),$ 
\[\begin{aligned} \mrD^{-1}\cS_\mrD P_1&=\int_{\sigma(\mrD)}
\frac{\cS(\rho)}{\rho} \mbbE_\rho(P_1)\,\mrd \rho, \\
&= \frac{\cS(\lambda_\mrD)}{\lambda_{\mrD}} \mbbE_{\lambda_\mrD}(P_1)+\int_{\rho>\omega}
\frac{\cS(\rho)}{\rho} \mbbE_\rho(P_1)\,\mrd \rho.
\end{aligned}
\]
Since $\cS$ is increasing and $\omega>0$ we have the spectral representation for the inner product
\beq\label{e:DinvC_BLF}
\begin{aligned}
\langle \mrD^{-1}\cS_\mrD P_1,P_1\rangle&=
\frac{\cS(\lambda_\mrD)}{\lambda_{\mrD}} 
|\hat P_1(\lambda_\mrD)|^2+
\int_{\rho>\omega}
\frac{\cS(\rho)}{\rho} |\hat P_1(\rho)|^2\,\mrd \rho,\\
&\geq \cS(\lambda_\mrD) \left( 
\frac{1}{\lambda_{\mrD}} 
|\hat P_1(\lambda_\mrD)|^2+
\int_{\rho>\omega}
\frac{1}{\rho} |\hat P_1(\rho)|^2\,\mrd \rho
\right),\\
&=\cS(\lambda_D) \langle\mrD^{-1}P_1,P_1\rangle.
\end{aligned}\eeq
By assumption $\langle \mrD^{-1}P_1,P_1\rangle>0$, and bounding $\cS(\lambda_\mrD)$ from below by $\beta_-$ yields the lower bound on $\Re\lambda.$
\end{proof}

We use the second tool to control the real point spectrum.  In general an operator with small negative point spectrum has an inverse iwth large negative spectrum. It is hard to make the constrained inverse operator positive. However in the case at hand the constraint space $S_\lambda\sim \{\psi_\mu\}^\bot$ for $|\lambda_\mrD|\ll 1$. This scaling leads the strength of the constraint to dominate the size of $\lambda_\mrD^{-1}$, and instills positivity of the constrained inverse.

  \begin{lemma}
  \label{l:n(Dinv)}
  There exists $\mu_*>0$ such that $n(\mrD^{-1}\bigl|_{S_\lambda})=0$ for all $\lambda\in\mbbR$ with $\lambda>-\beta_-$ and all $|\mu|\leq \mu_*.$ In particular, for these $\mu$ and $\lambda$ we have
  $\langle \mrD^{-1}P_1,P_1\rangle\geq 0$  and
  \beq\label{e:b/a} \frac{\langle \mrD^{-1}\cS_\mrD P_1,P_1\rangle}{\langle \mrD^{-1}P_1,P_1\rangle}\geq \beta_-,
  \eeq
  for all $P_1\in S_\lambda.$
  \end{lemma}
  \begin{proof}
   If $\lambda_\mrD>0$ then $\mrD^{-1}>0$ and the statement is trivial. Assume, $\lambda_\mrD<0$ but with $|\lambda_\mrD|$ sufficiently small, controlled by the size of $\mu_*.$  
   The spectrum of $\mrD^{-1}$ accumulates at the origin from the right. 
   To apply Proposition 5.3.1 of \cite{bib:KP13}, see also \cite{bib:KP12}, we shift the inverse operator slightly to open a spectral gap around zero. For $\lambda\in\mbbR$ the space $S_\lambda$ is real. Fix $\mu\neq 0$, the operator  $\mrD^{-1}+\delta$ has well-defined negative and zero indices, satisfying $n(\mrD^{-1}+\delta)=n(\mrD)=1$ for $\delta$ sufficiently small. When constrained to act on $S_\lambda$ we deduce that
\be \label{e:Dinv-index} n\left((\mrD^{-1}+\delta)\bigl|_{S(\lambda)}\right)=n(\mrD^{-1}+\delta)-n(\mrA),\ee
where 
\beq\label{e:Adef}
\begin{aligned}
\mrA&:=\langle (\mrD^{-1}+\delta)^{-1} S_\lambda,S_\lambda\rangle,\\
&=\langle \frac{\mrD}{1+\delta \mrD}\mrD^{-1}(\cS_\mrD+\lambda)\phi',\mrD^{-1}(\cS_\mrD+\lambda)\phi'\rangle,\\
&= \langle (1+\delta \mrD)^{-1}(\cS_\mrD+\lambda)\phi',\mrD^{-1}(\cS_\mrD+\lambda)\phi'\rangle.
\end{aligned}
\eeq
Recalling that $\psi$ is the ground state of $\mrD$ with eigenvalue $\lambda_\mrD(\mu)$, we write 
\be 
\label{e:phi-bot}
\phi'=\frac{\langle \phi',\psi\rangle}{\lVert\phi'\rVert\lVert\psi\rVert} \psi +\psi^\bot,
\ee
where $\psi^\bot\in \{\psi\}^\bot$ satisfies $\lVert\psi^\bot\rVert\leq \lVert\phi'\lVert.$ 
Since $\phi'$ and $\psi$ are non-negative ground states we have $$\langle\phi',\psi\rangle\neq0.$$
Since all operators in the inner-product defining $\mrA$ are functions of $\mrD$ they preserve the orthogonality of $\psi$ and $\psi^\bot.$ In particular, we have
\beq\label{e:A2}
\begin{aligned}
\mrA&= \frac{(\cS(\lambda_\mrD)+\lambda)^2}{\lambda_\mrD(1+\delta \lambda_\mrD)}\frac{|\langle\phi',\psi\rangle|^2}{\Vert \psi\Vert^2 \Vert\phi'\Vert^2}+\langle (1+\delta \mrD)^{-1}(\cS_\mrD+\lambda)\psi^\bot,\mrD^{-1}(\cS_\mrD+\lambda)\psi^\bot\rangle.
\end{aligned}
\eeq
Crucially $\psi^\bot$ has zero projection onto $\psi$ and the second term in \eqref{e:A2} is bounded for fixed $\lambda$ uniformly for all $|\mu|<\mu_*.$ For $\lambda>-\beta_+$ and $\delta$ sufficiently small we have $(\cS(\lambda_\mrD)+\lambda)^2/(1+\delta\lambda_\mrD)>0$ and the first term in \eqref{e:A2} scales with $|\lambda_\mrD^{-1}|\gg1$. In particular $A$ has the same sign as $\lambda_\mrD$ for $|\mu|\leq \mu_*$ with $\mu_*$ sufficiently small. For large values of $\lambda>0$ both terms in \eqref{e:A2} grow quadratically in $\lambda$, but the first term has the dominant coefficient of $\lambda^2$. By assumption $\lambda_\mrD<0$, hence we deduce that $n(A)=1$ and from \eqref{e:Dinv-index} that $n((\mrD^{-1}+\delta)\bigl|_{S_\lambda})=0.$
This implies that $\mrD^{-1}\cap[\infty,-\delta),$ is empty for all $\delta$ sufficiently small, equivalently $n(\mrD^{-1})=0$ and $\langle \mrD^{-1}P_1,P_1\rangle\geq 0$ for all $P_1\in S_\lambda.$ Since the $\mrD^{-1}$ bilinear form is positive, the estimate \eqref{e:b/a} follows from the inequality \eqref{e:DinvC_BLF} in Lemma\,\ref{l:complex-bound}. 
  \end{proof}

We use the linearity of the constraint $P_1\in S_\lambda$ to bound the ratio of bilinear forms from below.
 \begin{lemma}
\label{l:DDC-lower}
The control $\mu_*$ can be chosen small enough that
\beq\label{e:DDC-lower} 
\frac{\langle \mrD^{-1} P_1, P_1\rangle}{ \langle\mrD^{-1}\cS_\mrD P_1,P_1 \rangle} \geq \frac{1}{2\beta_+}
\eeq
for all $P_1\in S_\lambda$, all $\lambda>-\frac{\beta_-}{2},$ and all $|\mu|\leq \mu_*.$
\end{lemma}
\begin{proof}
From Lemma\,\ref{l:n(Dinv)} we have $\mrD\bigl|_{S_\lambda}>0$ for all $\lambda>-\frac{\beta_-}{2}.$ Assume $\lambda_\mrD<0$ and return to the integral representation for the bilinear form in the first line of \eqref{e:DinvC_BLF}. Since the $\lambda_\mrD$ term is negative it may be dropped to yield an upper bound. Using the constraints \eqref{e:cC-bds} on the map $\cS$yields the upper bound
\beq\label{e:LTUB}\begin{aligned}
\langle \mrD^{-1}\cS_\mrD P_1,P_1\rangle&\leq
\int_{\rho>\omega}
\frac{\cS(\rho)}{\rho}|\hat P_1(\rho|^2\,\mrd \rho,\\
&\leq \beta_+ \int_{\rho>\omega}\frac{1}{\rho}|\hat P_1(\rho)|^2\mrd\rho\leq \beta_+\Vert P_1\Vert_{\omega}^2,
\end{aligned}
\eeq
where the $\omega$ semi-norm is defined as
\beq\label{e:omega-norm} \Vert P\Vert_\omega^2:=\int_{\rho\geq\omega}\frac{1}{\rho}|\hat P(\rho)|^2\,\mrd \rho=\Vert \pi_{\lambda_{\mrD}} \mrD^{-1} P_1\Vert^2.
\eeq
Here $\pi_{\lambda_\mrD}$ is the orthogonal projection off of $\psi$, the $\lambda_\mrD$ eigenvector of $\mrD$. To obtain a lower bound on the upper term in  \eqref{e:DDC-lower} we use $P_1\in S_\lambda$ to estimate
\[ 0 =\langle P_1,s_\lambda\rangle=\frac{\cS(\lambda_\mrD)+\lambda}{\lambda_\mrD} \hat P_1(\lambda_\mrD)\hat{\phi'}(\lambda_\mrD)+
\int_{\rho\geq\omega}\frac{\cS_\mrD(\rho)+\lambda}{\rho} \hat P_1(\rho)\hat{\phi'}(\rho)\,\mrd\rho.\]
Here $\hat{\phi'}(\lambda_\mrD)=\langle \phi',\psi\rangle > 0$ as both are non-negative ground states of self-adjoint operators. This affords the bound
\[\begin{aligned}
    |\hat P_1(\lambda_\mrD)| &\leq \frac{|\lambda_\mrD|}{(\cS(\lambda_\mrD)+\lambda)\langle \phi',\psi\rangle}\int_{\rho\geq\omega}\frac{\cS(\rho)+\lambda}{\rho} \hat P_1(\rho)\hat{\phi'}(\rho)\,\mrd \rho.
\end{aligned}\]
We have $\cS(\lambda)+\lambda>\beta_-+\lambda>\frac{\beta_-}{2}$.  Applying Young's inequality to the integral yields the estimate
\[ |\hat P_1(\lambda_\mrD)| \leq  \frac{|\lambda_\mrD|}{\langle \phi',\psi\rangle \omega} \frac{\beta_++\lambda}{\beta_-+\lambda}\Vert P_1\Vert_{\omega} \Vert \phi'\Vert_{\omega}\]
This bound on the projection of $P_1$ yields the estimate 
\[\begin{aligned}
    \langle \mrD^{-1} P_1,P_1\rangle &= \frac{|\hat P_1(\lambda_\mrD)|^2}{\lambda_\mrD} + \int_{\rho>\omega}\frac{1}{\rho}|\hat P_1(\rho)|^2\,\mrd \rho
    &\geq (C\lambda_D+1)\Vert P_1\Vert_\omega^2,
\end{aligned}\]
where $C$ maybe chosen independent of $|\mu|\leq \mu_*$ and $\lambda>-\beta_-/2.$
Taking $\mu_*$ small enough that $\lambda_\mrD(\mu)>-\frac{1}{2C}$ allows a lower bound on the upper term in \eqref{e:DDC-lower}, while \eqref{e:LTUB} yields an upper bound on the lower term. Dividing the two estimates cancels the factor of $\Vert P_1\Vert_\omega^2,$ and yields \eqref{e:DDC-lower}.
\end{proof}

\begin{thm}
\label{t:1D_coercive}
    Let $\mrD$ and $\mrC$ satisfy \eqref{e:CDspec} and $\cS$ satisfy \eqref{e:cC-bds}. Then there exists $\mu_*,\lambda_M>0$ such that for all $|\mu|\leq \mu_*$ the spectrum of $\mrL$ satisfies 
    \be \label{e:mrL-spec}
    \sigma(\mrL)\subset\{0\}\cup \{\Re\lambda<-\lambda_M\}.
    \ee
Moreover the kernel of $\mrL$ is simple.
\end{thm}
\begin{proof}
By Lemma\,\ref{l:Kernel_L}, the kernel of $\mrL$ is simple. By Lemmas\,\ref{l:Weyl} and \ref{l:complex-bound}  the essential spectrum and all genuinely complex eigenvalues of $\mrL$ are uniformly in the left-half complex plane.
We turn to the set $\sigma_p(\mrL)\cap[-\beta_-/2,\infty).$  Lemma\,\ref{l:n(Dinv)} 
implies that 
$$\frac{\langle \cS_\mrD \mrD^{-1} P_1, P_1\rangle}{\langle \mrD^{-1} P_1, P_1\rangle} \geq \beta_-.$$ 
to the  quadratic formula, \eqref{e:Holyland},   Taking $\mu_*$ small enough  Lemma\,\ref{l:DDC-lower} implies that 
$$\frac{\langle \mrD^{-1} P_1, P_1\rangle} {\langle \cS_\mrD \mrD^{-1} P_1, P_1\rangle}>\frac{1}{2\beta_+}.$$  
The operator $\mrC$ is positive modulo a simple kernel spanned by $\phi'$. The projection 
\[\begin{aligned}
    \langle \phi',s_\lambda\rangle &=\langle \phi',\mrD^{-1}(\cS_\mrD+\lambda)\phi'\rangle, \\
    &= \frac{\mrC(\lambda_\mrD)+\lambda}{\lambda_D}\frac{\langle \phi',\psi\rangle^2}{\Vert\phi'\Vert^2_{L^2}} +O(1), 
\end{aligned}
\]
is non-zero for all $\lambda>-\beta_-/2.$ 
By classical results there exists $c_+(\lambda,\mu)>0$ such that
\[\langle \mrC P_1, P_1\rangle>c_+\Vert P_1\Vert^2_{L^2},\]
for all $P_1\in S_\lambda.$ Moreover $c_+$ is continuous in $\lambda>-\beta_-/2$, and the space $S_\lambda$ converges to a limit as $\lambda\to \infty.$  We deduce that $c_+$ has a positive lower bound $\underline c_+$ on the set $(\lambda,\mu)\in[-\frac{\beta_-}{2},\infty)\times [-\mu_*,\mu_*].$
From \eqref{e:LTUB} we obtain the bound
\[\langle \mrD^{-1}\cS_\mrD P_1,P_1\rangle \leq \frac{\beta_+}{\omega} \Vert P_1\Vert^2_{L^2},\]
for all $\lambda>-\frac{\beta_-}{2}$ and all $|\mu|<\mu_*.$ Combining these estimates with \eqref{l:DDC-lower} affords the bound
\beq\label{e:Triple_bdd} \frac{\langle \mrD^{-1} P_1,P_1\rangle \langle \mrC P_1,P_1\rangle}{\langle \mrD^{-1}\cS_\mrD P_1,P_1\rangle^2} \geq \frac{\omega\underbar c_+}{2\beta_+^2}.
\eeq
From the quadratic formula \eqref{e:Holyland} we deduce that all $\lambda\in\sigma_p(\mrL)\cap [-\frac{\beta_-}{2},\infty)$ satisfy the bound
\beq\label{e:Re-point}
\lambda \leq - \frac{\beta_-}{2}\left(1- \Re\sqrt{1-2 \frac{\omega\underbar c_+}{\beta_+^2}}\right).\eeq
 Taking $\lambda_M$ to be the largest of this quantity, the essential spectrum bound, and the genuinely complex point spectrum bound yields \eqref{e:mrL-spec}. 
\end{proof}

\section{Curvature Driven Flow}
To derive the motion of front solutions of the Modal PNLS  we consider a smooth, closed interface 
$\Gamma=\{\gamma(s)\,\big| s\in[0,L]\}$ and introduce the local Frenet coordinates
\be\label{e:Frenet}
x= \gamma(s) + n(s) z/\eps,
\ee
where $n(s)$ is the unit outward normal to the curve $\Gamma$ at point $\gamma(s)$ and
 $z$ is signed, $\eps$-scaled distance to $\Gamma. $ 
The interface $\Gamma$ is far from self-intersection then the change of variables from $x=(x_1,x_2)$ to $(s,z)$ is well defined on a neighborhood of $\Gamma$. Indeed there exists $\ell>0$ such that the neighborhood contains all points $x\in\Omega$ whose scaled distance to $z(x)$ to $\Gamma$ satisfies $z(x)<\ell/\eps.$  
The curve $\Gamma$ breaks the domain $\Omega$ into inner and outer regions
\beq\label{e:Omega_pm}\begin{aligned}
    \Omega_{\rm inner}&:=\{x \,\bigl|\, |z(x)|\leq\ell/\eps\},\\
\Omega_{\rm outer}&:=\{x \,\bigl|\, |z(x)|>\ell/\eps\}.
\end{aligned}
\eeq
Since the front function decays exponentially to a constant, for the purpose of matched asymptotics with $\eps\ll 1$ the normal rays 
$\{x(s,z)\,\bigl|\, |z|\leq \ell/\eps\}$ eminating from $\gamma(s)$ can be taken to be infinitely long. The error in this approximation reduces to exponentially small terms that do not have impact.  This allows the definition of the quasi-steady front function $\Phi=\Phi(x;\gamma)$ defined on $\Omega$,
\be 
\label{e:Ansatz}
\Phi(x):=\begin{pmatrix} \phi(z(x)) \cr 0 \end{pmatrix}, 
\ee
with the understanding that $\Phi(x)$ transitions to the far-field values $(\pm 1,0)^t$  in the outer domain $\Omega_{\rm outer}$
The evolution of $\Phi$ is tracked via the  interface map $\gamma=\gamma(s,t)$ whose motion is prescribed by the normal velocity. 

We proceed by multiscale expansion, using Cartesian variables in the outer domain and Frenet variables in the inner domain. The outer expansion is trivial and the matching amounts constant far-field values as outlined above. The details are equivalent to those in \cite{bib:PR24} and omitted, with attention focused on the inner expansion.
In the Frenet coordinates in $\Omega_{\rm inner}\subset\mbbR^2$ the scaled Laplacian takes the form
\beq\label{e:LapBelt}\eps^2\Delta=\partial_z^2+\eps \kappa_0(s)\partial_z +\eps^2 (z\kappa_1(s)\partial_z+\Delta_s)+\eps^3(\Delta_{s,1}+z^2\kappa_2(s)\partial_z) +O(\eps^4).\eeq
Here $\Delta_s$ is the Laplace-Beltrami operator on the interface $\Gamma$ and $\kappa_0=\kappa_0(s,t)$ is its curvature.  The higher order curvatures satisfy  $\kappa_i=(-1)^i\kappa_0^{i+1}$, 
 see \cite{bib:HP15}[eqn (6.37)] and \cite{bib:DP}[eqn (2.8)] for details of this derivation.

To expand the residual vector field near the quasi-equilibrium front profile  we write the vector field from \eqref{e:MFPNLS} in factored form $\mrF(U)=\mrM(U)U$, and conduct inner expansions
\be \label{e:Resid} 
\mrF(U) = \mrF_0+ \eps \mrF_1+ \eps^2 \mrF_2+\eps^3 \mrF_3 +O(\eps^4).\ee
This requires  expansions of $U$,
\beq U=U_0+\eps U_1 +\eps^2 U_2 +\eps^3 U_3+O(\eps^4)
\eeq
where each $U_i=(p_i,q_i)^t$ and $\mrM(U),$
$$\mrM(U)= \mrM_0+\eps \mrM_1 + \eps^2\mrM_2+\eps^3 \mrM_3 +O(\eps^4).$$
The quasi-steady assumption leads to
\beq \label{e:U0}
U_0=\Phi(z) =\begin{pmatrix}\phi(z)\cr 0\end{pmatrix}.
\eeq
The construction of the higher order terms will impose a factored structure
\beq U_i(z,s,t)= \chi_i(s,t) \bar U_i(z)= \chi_i(s,t)\begin{pmatrix}
    \bar p_i(z) \cr \bar q_i(z) \end{pmatrix}.
\eeq
We will use the overbar to denote functions whose inner expansion depend upon $z$ alone.

\subsection{Sub-Operator Expansions}

Each term in the expansion of the matrix factor $\mrM$ of $\mrF$ has sub-operators,
\beq \mrM_i=\begin{pmatrix} 0 & \cN_{-,i} \cr
-\cN_{+,i} & -\cM_i \end{pmatrix}.
\eeq
we expand each sub-operator in turn. The operators $\cN_\pm$ defined in \eqref{e:cN_expression} have expansions formed from concatenations of the Frenet variable expansion of $\eps^2\Delta$ with the regular expansion of $g_\pm(|U|^2).$
For a general smooth function $g:\mbbR\mapsto\mbbR$, assuming the leading order term \eqref{e:U0} we have
\beq
\label{e:gexpansion}
\begin{aligned}
    g(|U|^2)&= g(\phi^2)+ \eps g'(\phi^2)(2\phi p_1)+ 
    \eps^2\left( g'(\phi^2)(2\phi p_2+p_1^2+q_1^2)+2 g''(\phi^2)\phi^2p_1^2 \right)+\\
&\hspace{-0.1in}\eps^3\left(2g'(\phi^2)(\phi p_3+p_1p_2+q_1q_2)
    +2g''(\phi^2)\phi p_1(2\phi p_2+p_1^2+q_1^2)+\frac43 g'''(\phi^2)\phi^3p_1^3\right)+O(\eps^4).
\end{aligned}    
\eeq
The coefficient of $\eps^i$ in the expansion will be denoted by $G_i(z;U)$ for $i=1, 2, 3, \ldots.$
Combining \eqref{e:gexpansion} with the Laplace-Beltrami expansion \eqref{e:LapBelt} affords a systematic expansion
$$\cN_\pm(U)=\cN_{\pm,0}+\eps \cN_{\pm,1}+\eps^2\cN_{\pm,2}+\ldots,$$
into sum of operators in $z$ and $s$,
\be
\label{e:cN-exp}
\ba{rcl}
\cN_{\pm,0}&=&-\partial_z^2+g_\pm(\phi^2),\\ 
\cN_{\pm,1}&=&-\kappa_0(s) \partial_z +G_{\pm,1}(z;U),\\
\cN_{\pm,2}&=& -\kappa_1(s) z\partial_z-\Delta_s +G_{\pm,2}(z;U),\\
\cN_{\pm,3}&=&-\kappa_2(s)z^2\partial_z
-\Delta_{s,1}+ G_{\pm,3}(z;U),\ea\ee


To expand $\cM:=\cS(\cN_-)$ requires expressions for the resolvent $(\lambda-\cN_-)^{-1}$ that appears in the Cauchy functional integral.
That is we expand the solution of
$$ (\lambda-\cN_-)W = V,$$
where $V$ is supported in the inner domain and has smooth dependence on $s$. This allows the operators $\eps^2\Delta_s$ to be treated as perturbations in a Dirichet expansion,
\[\begin{aligned} (\lambda-\cN_-)^{-1} &= (\lambda-\mrD - \eps \cN_{-,1}-\eps^2 \cN_{-,2}+O(\eps^3))^{-1}, \\
&= 
\left(\mrI + \eps R_\lambda \cN_{-,1} +\eps^2( R_\lambda \cN_{-,2} +(R_\lambda \cN_{-,1})^2)+O(\eps^3)\right)R_\lambda,\\
&= R_\lambda +\eps R_\lambda \cN_{-,1}R_\lambda +\eps^2\left(R_\lambda \cN_{-,2}R_\lambda + R_\lambda \cN_{-,1}R_\lambda \cN_{-,1}R_\lambda\right)+
O(\eps^3),
\end{aligned}\]
where $R_\lambda:=(\lambda-\mrD)^{-1}$ is the resolvent of $\mrD.$
Using the Cauchy integral formulation the operator $\cM$ admits the expansion
\beq\begin{aligned} \cM_0&=\cS_\mrD,\\
\cM_1&= \frac{1}{2\pi i} \int_C \cS(\lambda) R_\lambda(\kappa_0\partial_z-G_{-,1}) R_\lambda\, \mrd \lambda,\\
\cM_2&:= \frac{1}{2\pi i} \int_C \cS(\lambda) 
R_\lambda\left( \cN_{-,2} +  \cN_{-,1}R_\lambda \cN_{-,1}\right) R_\lambda\, \mrd \lambda.
\end{aligned}
\eeq
These terms will be further refined as the structure of $U_1$ and $U_2$ is revealed. However a crucial feature is that the Laplace-Beltrami operator in $\cN_{-,2}$ makes no contribution to $\cM_2$. This arises from the residue theorem. Since $\Delta_s$ has no $z$ dependence the Laplace-Beltrami operator factors out of the contour integral. As $\cS(\lambda)$ has no zeros on the spectrum of $\mrD$ and is analytic on the region within the contour $C$, the residue associated to the factored contour integral is zero. This leads to the operator equality
\beq \label{e:cM2_expression}
\frac{1}{2\pi i} \int_C \cS(\lambda) 
R_\lambda \Delta_s R_\lambda\, \mrd \lambda=
\Delta_s \frac{1}{2\pi i} \int_C \cS(\lambda) 
R_\lambda^2\, \mrd \lambda =0.
\eeq

\subsection{Residual Expansion}
We collect orders in the residual, defined as the vector field $\mrF$ evaluated at the quasi-steady ansatz $U$. At the leading order we have
\beq \label{e:FExp0}
\mrF_0(U)= \begin{pmatrix} 0 & \cN_{-,0} \cr
-\cN_{+,0}& -\cM_0 \end{pmatrix} \begin{pmatrix} \phi \cr 0 \end{pmatrix} =0,
\eeq
 by virtue of the quasi-steady ansatz \eqref{e:U0} with $\phi$ solving the front system \eqref{e:phi-eq} which is equivalent to $\cN_{+,0}\phi=0.$
At the first order in $\eps$ the residual reduces to
\beq \label{e:F1}
\begin{aligned}
\mrF_1(U)&= 
\begin{pmatrix} 0 & \cN_{-,0} \cr
-\cN_{+,0}& -\cM_0 \end{pmatrix} \begin{pmatrix} p_1 \cr q_1 \end{pmatrix}+ 
\begin{pmatrix} 0 & \cN_{-,1} \cr
-\cN_{+,1}& -\cM_1 \end{pmatrix} \begin{pmatrix} \phi \cr0 \end{pmatrix},\\
&= 
\begin{pmatrix} 0 & \cN_{-,0} \cr
-\cN_{+,0}& -\cM_0 \end{pmatrix} \begin{pmatrix} p_1 \cr q_1 \end{pmatrix}+ 
 \begin{pmatrix} 0 \cr \kappa_0 \phi' -2g_+'(\phi^2)\phi p_1 \end{pmatrix},\\
 &= \mrL \begin{pmatrix} p_1 \cr q_1 \end{pmatrix} +\begin{pmatrix} 0 \cr \kappa_0 \phi' \end{pmatrix},
\end{aligned}
\eeq
where the linear operator $\mrL$ given in \eqref{e:mrL-def} arrives after collecting all terms in $U_1$. The invertability of $\mrL$ on inhomogeneities with a tensor product factorization is established in section 2. The structure of $\mrL^{-1}$ given in \eqref{e:Linv-gen} will propagate the tensor product structure to the correction terms.  At the second order in $\eps$ the residual becomes
\beq \label{e:F2}
\begin{aligned}
\mrF_2(U)&= 
\begin{pmatrix} 0 & \cN_{-,0} \cr
-\cN_{+,0}& -\cM_0 \end{pmatrix} \begin{pmatrix} p_2 \cr q_2 \end{pmatrix}+ 
\begin{pmatrix} 0 & \cN_{-,1} \cr
-\cN_{+,1}& -\cM_1 \end{pmatrix} \begin{pmatrix} p_1 \cr q_1 \end{pmatrix}+
\begin{pmatrix} 0 & \cN_{-,2} \cr
-\cN_{+,2}& -\cM_2 \end{pmatrix} \begin{pmatrix} \phi \cr0 \end{pmatrix}
\\
&= 
\begin{pmatrix} 0 & \cN_{-,0} \cr
-\cN_{+,0}& -\cM_0 \end{pmatrix} \begin{pmatrix} p_2 \cr q_2 \end{pmatrix}+ 
 \begin{pmatrix} -\kappa_0 \partial_z q_1+2g_-'(\phi^2)\phi p_1q_1 \cr 
 \kappa_0 \partial_zp_1 - 2g_+'(\phi^2)\phi p_1^2 -\cM_1 q_1 \end{pmatrix}+\\
 &\hspace{0.5in} \begin{pmatrix} 
 0 \cr\kappa_1 z\phi'- g_+'(\phi^2)(2\phi p_2+p_1^2+q_1^2)\phi -2g_+''(\phi^2) p_1^2\phi^3 \end{pmatrix},\\
 &= \mrL \begin{pmatrix} p_2 \cr q_2 \end{pmatrix} +\mrR_2(U_1), 
 \end{aligned}
 \eeq
 where we have introduced the second order inhomogeneity
 \beq\label{e:R2}
 \mrR_2(U_1):=\begin{pmatrix} 
 -\kappa_0 \partial_z q_1+2g_-'(\phi^2)\phi p_1q_1 \cr 
 \kappa_0 \partial_zp_1 - 2g_+'(\phi^2) \phi p_1^2 -\cM_1 q_1 +
 \kappa_1 z\phi'- g_+'(\phi^2)\phi |U_1|^2 -2g_+''(\phi^2) p_1^2\phi^3  \end{pmatrix}.
\eeq
The third-order expansion becomes cumbersome, we defer the calculations until the relevant sub-parts are identified and write the result symbolically as
\beq\label{e:F3}
F_3(U)= \mrL \begin{pmatrix} p_3 \cr q_3 \end{pmatrix} + \mrR_3,
\eeq
where the inhomogeneity takes the form
\beq\label{e:R3}
R_3(U_1,U_2) = 
\begin{pmatrix}  \cN_{-,1}q_2+\cN_{-,2}q_1 
\cr
-\cN_{+,1}p_2 -\cM_1q_2 
-\cN_{+,2}p_1 -\cM_2q_1
-\cN_{+,3}\phi \end{pmatrix}\Bigl|_{U_3=0}.
\eeq

\subsection{Normal Velocity and Matching}
The quasi-steady reduction of the system \eqref{e:MFPNLS}
in the Frenet variables involves an expansion of the normal variable $z=z(t)$ where the normal front velocity $V$ at a point $x$ in space is written through $z=z(x)$ as $\partial_t z=-\mrV.$ 
To extract the curvature dynamics we develop a quasi-steady manifold $U$ parameterized by the interface $\Gamma$ through the scaled distance function $z$ and the curvature $\kappa_0$. These quantities evolve on the slow time $T=\eps^2 t$ for which $\eps^2\partial_T=\partial_t$. 
The chain rule gives a material derivative 
\be 
\label{e:chain-rule}
D_T U=\frac{\partial U}{\partial z}\frac{\partial z}{\partial T}+\frac{\partial U}{\partial T}.
\ee
The normal velocity $V$ of the curve is scaled as   
$V:=-\eps^{-1} \frac{\partial z}{\partial T}.$ This affords the reduction
\be \label{e:T}
\partial_\tau \tU=\eps^2D_T U=-\eps V\frac{\partial U}{\partial z}+\eps^2\frac{\partial U}{\partial T}.\ee
The normal velocity admits a formal regular perturbation expansion
$$\mrV=\mrV_1+\eps \mrV_1+\eps^2\mrV_2+O(\eps^3),$$  
while the $T$ partial derivatives of $U$ satisfy
 $$\partial_T U= \eps \partial_TU_1+ \eps^2 \partial_T U_2+O(\eps^3).$$
 Combining these expansions yields the Frenet-variable inner expansion of the time derivative in \eqref{e:MFPNLS},
 \be \label{e:DT_exp}\begin{aligned}
  \partial_t U&=-\eps \mrV_0\partial_z U_0-\eps^2(\mrV_0\partial_z U_1+\mrV_1\partial_z U_0)- \\
  &\hspace{0.3in}\eps^3(\mrV_0\partial_z U_2+ \mrV_1\partial_z U_1 +\mrV_2\partial_z U_0s- \partial_T U_1)+O(\eps^4).
  \end{aligned}
\ee
The quasi-steady front normal velocity $\mrV$  is derived by balancing the temporal expansion \eqref{e:DT_exp} with the residual expansion terms from \eqref{e:F1}-\eqref{e:F3}. This is subject to the outer solution matching, which reduces to $U_i=0$ as $z\to\pm\infty$ for $i=1, 2, \ldots$

At $O(\eps)$, the balance in the evolution system  yields
\begin{equation}
 -\mrV_0 \begin{pmatrix}\phi'\cr 0\end{pmatrix}= \mrL U_1 + \mrR_1,
\eeq
which can be rearranged into the linear system for $U_1$
\begin{equation}
 \mrL U_1 =-\begin{pmatrix}\mrV_0 \\ \kappa_0 \end{pmatrix}\phi'.
\end{equation}
The inner-outer matching requires decay of the inner solution at infinity. Thus the  solvability is the usual Fredholm orthogonality to the adjoint kernel given by \eqref{e:Psi0},
\begin{equation}
\begin{pmatrix} \mrV_0 \\ \kappa_0 \end{pmatrix}\phi'\perp\Psi_0^\dag= \begin{pmatrix} \mrD^{-1}\cS_\mrD\phi' \\ \phi' \end{pmatrix}.
\label{e:solvability_V0}
\end{equation}
Both $\mrV_0$ and $\kappa$ are functions of the tangential variable $s$ that parameterizes location on the interface. They are constant in $z$, and the $L^2(\mbbR_z)$-orthogonality condition connects the leading order normal velocity $\mrV_0$ to the curvature
\begin{equation}
\mrV_0 =  -\alpha_1\kappa_0,
\label{e:V0_result}
\end{equation}
where the negative sign is chosen to align sign$(\alpha_1)$ with sign$(\mu)$ and to have positive $\alpha_1$ correspond to motion by curvature. This coupling constant has the definition
\begin{equation}
\label{e:alpha0}
\alpha_1(\mu):= \frac{\lVert\phi'\rVert^2}{\langle \mrD^{-1}\cS_\mrD\phi', \phi'\rangle}.
\end{equation}
The assumptions  the spectrum of $\mrD$, and the analysis of Lemma\,\ref{l:n(Dinv)}, establishes that 
\[\alpha_1(\mu)=\frac{\lambda_\mrD \Vert\phi'\Vert^2}{\cS(\lambda_\mrD)\langle \phi',\psi\rangle^2} +O(\mu^2),\]
for $|\mu|\ll1$. Since $\cS(\lambda_D)\geq \beta_->0$ this shows that sign$(\alpha_1)=\,$sign$(\lambda_d)=\,$sign$(\mu)$. In particular all three change sign at $\mu=0$. The former is motion by curvature and the latter is motion against curvature, which requires regularization to be locally well posed. To address this requires resolution of the higher order terms. 

The first step is to solve for $U_1$. Using the inverse formula \eqref{e:Linv-gen} we have
\beq\label{e:U1}
U_1 = 
\kappa_0\mrL^{-1} \begin{pmatrix} \alpha_1 \phi' \cr - \phi' \end{pmatrix}
=
\kappa_0\begin{pmatrix}
 \mrC^{-1} \left(\phi'-\alpha_1\mrD^{-1}\cS_\mrD \phi'\right) 
 \cr
\alpha_1 \mrD^{-1}\phi'
\end{pmatrix}=
\kappa_0 \begin{pmatrix}\alpha_1 \mrC^{-1} \Pi_{\phi'}^\bot \mrD^{-1}\cS_\mrD \phi' \cr \alpha_1\mrD^{-1}\phi' \end{pmatrix},
\eeq
where $\Pi_{\phi'}^\bot$ is the orthogonal projection onto 
$\{\phi'\}^\bot=\{\ker(\mrC)\}^\bot.$
This motives the tensor factorization $U_1(s,z)= \kappa_0(s) \bar U_1(z),$
where
\beq\label{e:bU1}
\bar U_1:= \begin{pmatrix} \bar p_1 \cr \bar q_1 \end{pmatrix} =  \begin{pmatrix} \alpha_1 \mrC^{-1} \Pi_{\phi'}^\bot \mrD^{-1}\cS_\mrD \phi' \cr \alpha_1 \mrD^{-1}\phi' \end{pmatrix}.
\eeq
Here and below we use an overbar to denote functions of $z$ alone, in particular the barred functions are independent of time. Parity allows simplification of the calculations. Since $\phi$ has odd parity, both $\phi'$ and $\phi^2$ have even parity and the operators $\mrC$, $\mrD$, and their functional relations preserve parity. We deduce that $\bar U_1=(\bar p_1, \bar q_1)^t$ has even parity about z=0.

At order of $\eps^2$  the evolution system yields the balance
\[ -\mrV_0 \partial_z U_1-\mrV_1\partial_z U_0 = \mrL \begin{pmatrix} p_2 \cr q_2 \end{pmatrix} +\mrR_2. \]
To break $\mrR_2$ into terms of even and odd parity we revisit $\cM_1$, which in light of \eqref{e:U1} reduces to
\beq\label{e:cM1-expression}\begin{aligned} \cM_1 &= \kappa_0 \frac{1}{2\pi i}\int_C \cS(\lambda)\left(R_\lambda \partial_z R_\lambda-2R_\lambda g_-'(\phi^2)\phi \bar p_1 R_\lambda\right)\,\mrd \lambda, \\
&= \kappa_0 \bar\cM_1.
\end{aligned}\eeq
Here the bar denotes that the operator acts only in the $z$ variables. The operator $\bar\cM_{1}$ flips parity.
We decompose the second order residual inhomogeneity into a tensor product
\beq \label{e:R2TP}
\mrR_2(s,z,t) =\kappa_0^2(s,t) \bar\mrR_2(z),
\eeq
where the $z$ dependent term has odd parity with respect to $z,$
\beq\label{e:barR2}
\bar\mrR_2:=\begin{pmatrix} 
 -  \bar q_1'+2g_-'(\phi^2)\phi \bar p_1\bar q_1 \cr 
  \bar p_1' - 2g_+'(\phi^2) \phi \bar p_1^2 -\bar\cM_1 \bar q_1 -
  z\phi'- g_+'(\phi^2)\phi |\bar U_1|^2 -2g_+''(\phi^2) \bar p_1^2\phi^3  \end{pmatrix}.
 \eeq
 We  rearrange the second order balance into a linear system for $U_2,$
\beq\label{e:U2a}\mrL \begin{pmatrix} p_2 \cr q_2 \end{pmatrix}= -\mrV_1 \begin{pmatrix} \phi' \cr 0\end{pmatrix} +\kappa_0^2\left(\alpha_1 \bar U_1'-
\bar\mrR_{2}\right).
\eeq
Since both $\bar U_1' and \bar \mrR_2$ have odd parity while $\Psi_0^\dag$ has even parity, the solvability condition for $U_2$ is satisfies with $\mrV_1=0.$
The $U_2$ corrections take the form
\beq\label{e:U2b}
\begin{pmatrix} p_2 \cr q_2 \end{pmatrix}
=\kappa_0^2 \begin{pmatrix} \bar p_2 \cr \bar q_2 \end{pmatrix},
\eeq
where the $z$-only function
\beq\label{e:U2c}
\bar U_2 =\begin{pmatrix} \bar p_2 \cr \bar q_2 \end{pmatrix}= \mrL^{-1} \left(\alpha_1 U_1'-\bar \mrR_{2}\right),\eeq
has odd parity. 

\subsection{The Wilmore Flow}
The normal velocity obtains a singular regularization at $O(\eps^2).$  
At this order the evolution system yields the balance
\beq \label{e:ThirdOrder}\partial_T U_1 - \mrV_0\partial_z U_2  -\mrV_2\partial_z U_0 = \mrL U_3 +\mrR_3. \eeq
From \eqref{e:U1}, the function $U_1$ is a tensor product of the interface curvature with a time-independent function of $z$ alone. We deduce that
\[ \partial_T \mrU_1= (\partial_T\kappa_0 )\bar \mrU_1.\]
 From Pismen, \cite{bib:Pismen} in two space dimensions in co-moving coordinates the curvature evolution is dependent upon the normal velocity,
\be \label{e:Pismen}
\partial_T\kappa_0=-(\Delta_s+\kappa_0^2)\mrV=\alpha_1\left(\Delta_s\kappa_0+\kappa_0^3\right)+ O(\eps^2).
\ee
This allows $\partial_T U_1$ \eqref{e:ThirdOrder} to be replaced by the leading order terms on the right-hand side of \eqref{e:Pismen}. Using the expressions for $\mrV_0$ and for $U_0$, $U_1,$ and $U_2$ we find
\[ \mrL U_3=- \alpha_1 \Delta_s\kappa_0\mrU_1 +\alpha_1\kappa_0^3\left(- \bar U_1 + \bar U_2' \right)-\mrV_2 \begin{pmatrix}\phi' \cr 0 \end{pmatrix} -\mrR_3.\]
The solvability condition for $U_3$ determines the $O(\eps^2)$ correction to the normal velocity
\[ \mrV_2= -\frac{\alpha_1}{\Vert\phi'\Vert_{L^2}}\left(\left\langle\mrR_3,\Psi_0^\dag\right\rangle +\alpha_1\Delta_s\kappa_0\left \langle U_1,\Psi_0^\dag\right\rangle +\alpha_1 \kappa_0^3 \left\langle \bar U_1-\bar U_2',\Psi_0^\dag\right\rangle \right).\]
The well-posedness of the curvature evolution depends upon the sign of the coefficient of the curvature surface diffusion term, $\Delta_s \kappa_0$, in the normal velocity. 
Within $\mrR_3$ the surface Laplacian arises in $\cN_{\pm,k}$ with $k\geq 2.$ However the operator $\cN_{+,3}$ acts on $\phi$ for which $\Delta_s\phi=0.$  Similarly $\cM_1$ given in \eqref{e:cM1-expression} has no surface differential terms, while $\cM_2$  has a zero contribution from its surface diffusion term, see \eqref{e:cM2_expression}. Extracting the curvature surface diffusion terms from $\mrR_3$ yields 
\beq\label{e:R3b}
\mrR_3(U_1,U_2) = 
\Delta_s \kappa_0 \begin{pmatrix}  - \bar q_1 
\cr 
 \bar p_1 \end{pmatrix} + \kappa_0^3\, \bar\mrR_{3,1}(\bar U_1,\bar U_2),
\eeq
where the $\bar\mrR_{3,1}$ terms are functions of $z$ alone. The result yields a second correction to the normal velocity in the form
\beq\label{e:V2}
\mrV_2= \nu\Delta_s\kappa_0 +\alpha_3 \kappa_0^3, 
\eeq
where the coefficient $\nu$ is given by
\beq\label{e:nu}
\nu = -\frac{\alpha_1}{\Vert\phi'\Vert^2_{L^2}} \left \langle \begin{pmatrix}
-\bar q_1+\alpha_1 \bar p_1 \cr \bar p_1+\alpha_1\bar q_1\end{pmatrix}  , \Psi_0^\dag \right\rangle.
\eeq 
Including the $\mrV_2$ terms, the normal velocity obtains the final form reported in \eqref{e:nV_expression}. 

 The sign of $\nu$ is essential to the well-posedness of the truncated normal velocity system. Indeed, the inclusion of $\mrV_2$  in the curvature flow \eqref{e:Pismen} yields the system 
\be \label{e:Pismen2}
\partial_T\kappa_0=\left(\alpha_1\Delta_s -\nu\eps^2 \Delta_s^2\right) \kappa_0 + \alpha_1 \kappa_0^3 -\eps^2 \left(\alpha_3(\Delta_s+\kappa_0^2)\kappa_0^3 +\nu\kappa_0^2\Delta_s\kappa_0\right) +O(\eps^3).
\ee
The $O(\eps^3)$ error terms are all bounded relative to the operator $1+\Delta_s^2$, see \cite{bib:HP15}[Section 6]. Consequently the system is locally well-posed and regularized if $\nu>0$, independent of $\eps\ll1.$ In this case the dominant aspects of the front evolution are controlled by the first three terms. 
Applying the expansions \eqref{e:Dinv} to the adjoint eigenvector $\Psi_0^\dag$, $\alpha_1$, and $\bar U_1$ given in \eqref{psi0mu0},  \eqref{e:alpha0} and \eqref{e:bU1}, respectively yields,
\be \label{e:nu1}
\nu = \frac{1}{ \langle \phi',\psi\rangle^2 \cS(\lambda_\mrD)} +O(\mu).
\ee
The inner product $\langle \phi',\psi\rangle\neq 0$ as both $\phi'$ and $\psi$ are non-negative ground state eigenvectors while $\psi$ is normalized to have $\Vert \psi\Vert_{L^2}=1$. The assumption \eqref{e:cC-bds} keeps $\cS(\lambda_\mrD)$  positive and bounded from  above. 
We deduce that $\nu$ is strictly positive for $|\mu|\leq \mu_*$ for sufficiently small $\mu_*.$ This establishes Main Result\,\ref{MR:1}.

\section{Discussion}

It is possible to extend the range of spectral maps $\cS$ to include growth at infinity, so that the operator $\cM$ is unbounded. One can also design maps that inhibit the curve lengthening bifurcation while preserving the linear stability of the front.   If the map $\cS$ grows linearly at infinity, then $\cS\mrD^{-1}$ is bounded but not smoothing.  The operator $\cL$ remains a compact perturbation of $\cL_\infty$, and the essential spectrum can be controlled. However to maintain   the spectral gap at $\lambda=0$, requires a different strategy to obtain the lower bound \eqref{e:Triple_bdd}. In particular 
$\mrD^{-1}\cS_\mrD$ resembles a multiple of the identity on high frequency terms and maintaining the spectral gap requires an estimate of the form
\beq\label{e:Triple_bdd2} \frac{\langle \mrD^{-1} P_1,P_1\rangle \langle \mrC P_1,P_1\rangle}{\Vert P_1\Vert^2_{L^2}} \geq \alpha,
\eeq
for some $\alpha>0$ and all $P_1\in S_\lambda.$  There is no clear compactness and a lower bound is likely not attained. Estimates of this type require coordination between the operators $\mrD$ and $\mrC$ on high frequency spaces.  If $\cS$ grows faster than linear then $\cL$ may not be a compact perturbation of $\cL_\infty$ and the control of the essential spectrum may break down.

On the other hand, it is tempting to design $\cS$ so that it inhibits the curve lengthening bifurcation. The logical choice is to take $\cS$ positive on $\mbbR_+$ and  negative on $\mbbR_-$ and uniformly bounded. In this setting it is straightforward to see that the essential spectrum remains controlled as the Fredhom boarder samples $\cS$ only where it takes uniformly positive values. For the point spectrum, $\cS_\mrD\mrD^{-1}$ is positive without constraint, and the steps to obtain an estimate of the form 
\eqref{e:Triple_bdd} are straightforward.
With this modification $\cS(\lambda)$ changes sign at $\mu=0$ and the linear normal velocity coefficient $\alpha_1$ will not change sign at $\mu=0,$ remaining positive. This may seem  to inhibit the bifurcation. However, the small $\mu$ expansions \eqref{e:Dinv} break down and the sign of $\nu$ becomes indefinite. Thus it is unclear if the interface evolution will be stable with respect high frequency perturbations of interface shape. 

\section*{Acknowledgment} 
The first author acknowledges NSF support through grant DMS 2205553.  The second author received funding through the SIAM 2024 Postdoctoral Support Program.

\end{document}